\documentclass[11pt]{article}
\usepackage{chao}
\usepackage{algo}
\usepackage{array}
\usepackage{booktabs}

\newcommand{\ind}{\mathcal{I}}
\newcommand{\F}{\mathcal{F}}
\newcommand{\cl}{\operatorname{cl}}
\newcommand{\abl}{\operatorname{abl}}
\newcommand{\one}{\mathbf{1}}
\newcommand{\gap}{\operatorname{gap}}
\newcommand{\supp}{\operatorname{supp}}

\title{Integrality-Gap Bounds for Weighted Matchoids and Matroid Intersection}
\author{Yu Cong\footnote{ \email{yucong143@gmail.com}, University of Electronic Science and Technology of China.} \and Yajie Zhao\footnote{ \email{yajiezhao8@gmail.com}, University of Electronic Science and Technology of China.}}
\date{}

\begin{document}
\maketitle
{\let\thefootnote\relax\footnotetext{\textbf{AI Disclosure.} We used GPT-5.6 Sol to assist with developing the integrality-gap proof and deriving the local-ratio algorithm. The tool materially affected Sections \ref{sec:gap-proof} and \ref{sec:algorithm}. The authors verified the correctness and originality of all content including references. }}

\begin{abstract}
The weighted $k$-matroid intersection problem asks for a maximum-weight set
that is independent in each of $k$ matroids on a common ground set.  The
natural LP relaxation optimizes over the intersection of the $k$ matroid
independent set polytopes.  It is conjectured that this LP has integrality gap
at most $k-1$.  The conjecture is known for $k\le3$, but for $k\ge4$ the best
general upper bound was $k$.  We improve this bound to
$k-1+1/k$.  More generally, we prove that the natural LP of a
$p$-matchoid has integrality gap at most $p-1+1/p$, with a deterministic
LP-relative algorithm attaining the same factor.  The matchoid extension
resolves the $p$-matchoid part of a conjecture of Lee, Sviridenko, and
Vondr\'ak; projective planes give explicit tight instances whenever one of
order $p-1$ exists.
\end{abstract}

\section{Introduction}

The \emph{weighted $k$-matroid intersection problem} asks for a
maximum-weight set that is independent in each of $k$ matroids on a common
finite ground set.  It lies at a sharp boundary in matroid optimization.  For
$k=2$, Edmonds' matroid-intersection theorem gives exact polynomial-time
optimization and an integral natural LP.  Already for $k=3$, however, even
the unweighted problem is APX-hard, the partition-matroid special case
contains three-dimensional matching, and the natural LP has integrality gap
exactly $2$
\cite{edmonds1979intersection,schrijver2003combinatorial,
      linhares2020approximate}.

For general $k$, the LP relaxation allows elements to be selected fractionally
while retaining all independent set polytope constraints of the input
matroids.  Its weighted integrality gap is the worst-case ratio between the
LP optimum and the maximum weight of a common independent set.  Unlike the
integral problem for $k\ge3$, this LP remains polynomial-time optimizable in
the standard matroid-oracle model
\cite{iwata2001combinatorial,grotschel1981ellipsoid,
      schrijver2003combinatorial}.  Its gap therefore measures how much of
the exact two-matroid polyhedral theory survives beyond the tractability
threshold. An upper bound of $g$ says that every fractional optimum can
lose at most a factor $g$ when replaced by a common independent set, whereas
a lower-bound instance rules out any better guarantee against this LP.  Since
intersections of partition matroids encode $k$-dimensional matching, the
question also extends fundamental fractional hypergraph-matching bounds to
general matroid constraints
\cite{furedi1993fractional,aharoni2025coloring}.

Aharoni and Berger proposed a matroidal extension of Ryser's conjecture with
factor $k-1$ \cite[Conjecture~7.1]{aharoni2006intersection}.  Its weighted
polyhedral analogue, formulated explicitly by Aharoni, Berger, Guo, and
Kotlar, asserts that the natural LP has integrality gap at most $k-1$
\cite[Conjecture~10.11]{aharoni2025coloring}.  For $k\ge4$, the best
known general upper bound was $k$.  Our main result gives the first
improvement on that bound.

\begin{theorem}\label{thm:main}
For every integer $k\ge2$, the natural LP for weighted $k$-matroid
intersection has integrality gap at most $k-1+1/k$.
\end{theorem}

The structural argument behind \autoref{thm:main} also yields an approximation algorithm with the same approximation factor.

\begin{corollary}\label{cor:algorithm-intro}
Given rational weights and independence or rank oracles for $M_1,\ldots,M_k$,
there is a deterministic polynomial-time $(k-1+1/k)$-approximation relative
to the natural LP.
\end{corollary}

The proof in fact depends on the number of matroid constraints containing
each element, rather than on the total number of matroids.  A
\emph{$p$-matchoid} is specified by matroids on possibly overlapping local
ground sets, with each element occurring in at most $p$ of them
\cite{huang2023matchoid}.  We obtain the following extension, proved in
\autoref{sec:matchoids}.

\begin{theorem}\label{thm:matchoid}
For every integer $p\ge2$, the natural LP of a $p$-matchoid has integrality
gap at most $p-1+1/p$.
\end{theorem}

\begin{corollary}\label{cor:matchoid-algorithm}
Given rational weights, the local ground sets, and independence or rank
oracles for their matroids, there is a deterministic polynomial-time
$(p-1+1/p)$-approximation relative to the natural $p$-matchoid LP.
\end{corollary}

Lee, Sviridenko, and Vondr\'ak conjectured that the relevant LP has gap
exactly $p-1+1/p$ for weighted $p$-matchoids
\cite[Conjecture~1]{lee2013matching}.  Their LP is stated as an extended
formulation obtained from a reduction to matroid $p$-parity.  In
\autoref{sec:matchoids} we describe their formulation and prove that its
projection is precisely our natural matchoid LP.  Since hypergraph matching
already supplies the matching lower bound \cite{chan2012hypergraph},
\autoref{thm:matchoid} resolves the $p$-matchoid part of their conjecture.
Their separate prediction of a $p-1$ gap for the intersection of $p$ global
matroids remains open; \autoref{thm:main} makes progress toward it.

\paragraph*{Previous results.}
The previous results, summarized in \autoref{tab:previous-results}, should
be read along two dimensions.  The third column concerns the integrality gap
of the natural LP, while the fourth records approximation algorithms.  An
LP-relative approximation compares its output directly with the natural LP
optimum and therefore proves the same upper bound on the integrality gap.  By
contrast, an algorithm analyzed only against the integral optimum may have a
better approximation factor without giving any information about the natural
LP.

For two arbitrary matroids, Edmonds' polyhedral theorem makes the natural LP
integral \cite{edmonds2003submodular,schrijver2003combinatorial}, while his
weighted matroid-intersection algorithm optimizes exactly
\cite{edmonds1979intersection}.  The first nontrivial weighted case is three
matroids.  Linhares, Olver, Swamy, and Zenklusen \cite{linhares2020approximate} proved that the natural LP
has gap exactly $2$. Their
iterative-refinement algorithm is a weighted LP-relative $2$-approximation.  This integrality-gap bound is tight.  It
remains open whether weighted $3$-matroid intersection admits an approximation
factor below $2$ \cite{singer2026ordered}.

The unweighted three-matroid case illustrates this distinction sharply.  The
natural LP still has gap $2$: the general unweighted upper bound of Lau,
Ravi, and Singh specializes to $2$ \cite{lau2011iterative}, and the
three-partite matching examples give the matching lower bound
\cite{furedi1981maximum}.  Nevertheless, Lee, Sviridenko, and Vondr\'ak gave
a $(3/2+\e)$-approximation for maximum-cardinality three-matroid
intersection \cite{lee2013matching}.  This stronger factor is not LP-relative.

For arbitrary $k$ and arbitrary matroids, the previous weighted LP-gap bounds
were $k-1\le \gap\le k$.  The upper bound is furnished by the LP-relative
greedy analysis \cite{korte1978greedy,linhares2020approximate}.  The lower
bound already occurs for weighted partition matroids: their natural LP has gap
$k-1$ \cite{furedi1993fractional,chan2012hypergraph}, with LP-relative
$(k-1)$-approximations
\cite{chan2012hypergraph,parekh2015generalized}.  In contrast, for the
unweighted objective, Lau, Ravi, and Singh proved the stronger LP-relative
upper bound $k-1$ \cite{lau2011iterative}; truncated projective planes attain
it whenever $k-1$ is a prime power \cite{furedi1981maximum}.
For weighted arbitrary-matroid intersection, Lee,
Sviridenko, and Vondr\'ak obtained factor $k-1+\e$
\cite{lee2010submodular}; Singer and Thiery later obtained the randomized
factor $(k+1)/(2\ln2)$ \cite{singer2025better}, and Feldman and Ward obtained
the randomized factor $(k+1)\ln2+O(\e)$
\cite{feldman2026multiplicative,feldman2026multiplicativefull}.  Most recently,
Singer and Thiery obtained factor $k/2+o(k)$, even for monotone submodular
objectives \cite{singer2026ordered}.  For the unweighted problem, Lee,
Sviridenko, and Vondr\'ak obtained $k/2+\e$
\cite{lee2013matching}.  A linear dependence on $k$ is unavoidable.  The
$k/(12+\e)$-hardness of Lee, Svensson, and Thiery
\cite{lee2025hardness} was strengthened by Minzer and Zheng to
$k/(8+\e)$ for sufficiently large $k$, unless
$\mathsf{NP}\subseteq\mathsf{BPP}$ \cite{minzer2026near}.

\begin{table}[htbp]
\centering
\small
\begin{tabular}{@{}
  >{\centering\arraybackslash}p{0.15\textwidth}
  >{\centering\arraybackslash}p{0.15\textwidth}
  >{\raggedright\arraybackslash}p{0.25\textwidth}
  >{\raggedright\arraybackslash}p{0.35\textwidth}@{}}
\toprule
\textbf{Number of matroids} & \textbf{Objective}
  & \textbf{LP integrality gap}
  & \textbf{Approximation guarantees} \\
\midrule
2
  & Weighted
  & $1$ \cite{edmonds2003submodular,schrijver2003combinatorial}
  & $1$ (exact) \cite{edmonds1979intersection} \\

3
  & Weighted
  & $2$ \cite{linhares2020approximate}
  & \emph{LP:} $2$ \cite{linhares2020approximate} \\

3
  & Unweighted
  & $2$ \cite{lau2011iterative,furedi1981maximum}
  & \emph{OPT:} $3/2+\e$ \cite{lee2013matching} \\

$k\ge3$
  & Unweighted
  & $k-1$ \cite{lau2011iterative,furedi1981maximum}
  & \emph{OPT:} $k/2+\e$
    \cite{lee2013matching} \\

\addlinespace[5pt]

$k\ge4$
  & Weighted
  & $\le k$ \cite{korte1978greedy,linhares2020approximate};\newline
       \textbf{$\le k-1+1/k$ (this work)}
  & \emph{OPT:} $k-1+\e$
    \cite{lee2010submodular};\newline
    $(k+1)/(2\ln2)$ randomized
    \cite{singer2025better};\newline
    $(k+1)\ln2+O(\e)$ randomized
    \cite{feldman2026multiplicative,feldman2026multiplicativefull};\newline
    $k/2+o(k)$ \cite{singer2026ordered}\\

\bottomrule
\end{tabular}
\caption{Previous bounds for $k$-matroid intersection.  ``LP'' denotes a
factor relative to the natural LP optimum, while ``OPT'' denotes a factor
relative only to the integral optimum.  Weighted results also apply to
unweighted objectives; separate unweighted rows record stronger
cardinality-specific guarantees.  Smaller approximation factors are better.}
\label{tab:previous-results}
\end{table}

The local-search (or local-search-based) algorithms of Lee, Sviridenko, and
Vondr\'ak \cite{lee2010submodular,lee2013matching} and of Singer and Thiery
\cite{singer2025better,singer2026ordered} displayed in the table run in time
polynomial in the explicit instance size for every fixed $k$ and fixed
accuracy parameters (in particular, fixed $\e$ for the ``$+\e$'' entries);
their neighborhood searches have $k$-dependent complexity.
In contrast, the Feldman--Ward algorithm
\cite{feldman2026multiplicative,feldman2026multiplicativefull} uses
constant-size exchanges and runs in $\operatorname{poly}(|E|,\e^{-1})$
oracle time with no $k$-dependent exponent.
The algorithm of this work is deterministic weakly polynomial in
the standard oracle model: for rational weights, its running time is
polynomial in $|E|$, $k$, the local-oracle input size, and the encoding
length of the weights.

Thus, before this work no weighted LP-relative factor below $k$ was known for
arbitrary matroids when $k\ge4$.  \autoref{thm:main} brings the general
weighted integrality-gap bound to within $1/k$ of the conjectured value $k-1$.

The intersection of partition matroids is precisely a capacitated
multipartite hypergraph-matching problem.  To see the unit-capacity case,
start with a $k$-partite $k$-uniform hypergraph and introduce one partition
matroid for each vertex class.  Its blocks group the hyperedges incident to
the same vertex.  The common independent sets are exactly the matchings, and
the natural matroid-intersection LP becomes the standard fractional matching
LP.  Allowing arbitrary block capacities gives $k$-partite hypergraph
$b$-matching \cite{chan2012hypergraph,parekh2015generalized}.

This correspondence explains the connection with Ryser's conjecture, which
asserts that the minimum vertex-cover size of a $k$-partite $k$-uniform
hypergraph is at most $k-1$ times its maximum matching size
\cite{aharoni2025coloring}.  LP duality identifies the fractional matching
and fractional vertex-cover optima.  F\"uredi proved the corresponding
fractional Ryser bound, and F\"uredi, Kahn, and Seymour proved its weighted
extension \cite{furedi1981maximum,furedi1993fractional}.  Thus, the $k-1$
integrality-gap conjecture extends the fractional weighted Ryser inequality
from partition matroids to arbitrary matroids.  It does not imply the full
Ryser conjecture, which concerns integral vertex covers.

For comparison, F\"uredi, Kahn, and Seymour proved the upper bound
$k-1+1/k$ for the standard matching LP of a general $k$-uniform hypergraph;
projective planes attain it whenever $k-1$ is a prime power
\cite{furedi1993fractional,chan2012hypergraph}.  The matchoid extension in
\autoref{thm:matchoid} explains the agreement: intersections of $k$ global
matroids and general $k$-uniform hypergraph matching are both special cases
of $k$-matchoids.  The former uses $k$ matroids on the common ground set,
whereas the latter uses one local rank-one matroid at each hypergraph vertex.
For the hypergraph special case, the local-ratio coefficient used below is
the coefficient of Anegg, Angelidakis, and Zenklusen for nonuniform
hypergraph matching and $b$-matching \cite{anegg2021simpler}.  More broadly,
other choices of matroids model such settings as matroid-constrained spanning
trees and matroid-intersection covers
\cite{linhares2020approximate,im2021cover}.

\paragraph*{Organization.}
Section \ref{sec:prelim} introduces the natural LP and the common independent
set polytope.  Section \ref{sec:gap-proof} proves the integrality-gap bound.
Section \ref{sec:algorithm} develops the fractional local-ratio approximation
algorithm.  Section \ref{sec:matchoids} extends both results to
$p$-matchoids and gives matching lower-bound examples.

\section{Preliminaries}
\label{sec:prelim}

For a vector $v\in\R^E$ and $S\subseteq E$, write
$v(S):=\sum_{e\in S}v_e$.  The incidence vector of $S\subseteq E$ is
$\one_S\in\{0,1\}^E$.
We follow \cite{oxley2011matroid} for notation in matroid theory.
For a matroid $M=(E,\ind)$, its independent
set polytope is
\[
 P(M):=\conv\{\one_I:I\in\ind(M)\}
 =\{y\in\R_{\ge0}^{E}:y(S)\le r_M(S)\text{ for every }S\subseteq E\}.
\]
Then the weighted $k$-matroid intersection problem can be formulated as the
integer program
\begin{equation*}
\begin{aligned}
\max& & w\cdot x& \\
s.t.& & x&\in P(M_i) &  &\forall i\in [k]\\
    & & x&\in \set{0,1}^E
\end{aligned}
\end{equation*}
Thus the LP relaxation feasible region is
$P:=\bigcap_{i=1}^kP(M_i)$.  The family of sets independent in all $k$
matroids and its common independent set polytope are
\[
 \F:=\bigcap_{i=1}^k\ind_i
   =\{I\subseteq E:I\in\ind_i\text{ for every }i\},
\]
and $Q:=\conv\{\one_I:I\in\F\}$, respectively.
We use the standard closure $\cl_M(S):=\{e\in E:r(S\cup\{e\})=r(S)\}$.
Matroids are given in the rank oracle model.

\section{The Integrality-Gap Bound}
\label{sec:gap-proof}

We prove \autoref{thm:main} through an equivalent polyhedral containment.
The family $\F$ is downward closed, and so are $P$ and $Q$.  It suffices to
prove $P\subseteq\alpha Q$ for $\alpha:=k-1+\frac1k$.

To test the containment, we measure the smallest dilation of $Q$ containing
a point $x\in\R_{\ge0}^E$.  Define $\gamma_Q(x):=\inf\{t\ge0:x\in tQ\}$.
The convex-hull description of $Q$, together with downward closedness, turns
this quantity into the covering LP
\begin{equation}
 \gamma_Q(x)=\min\left\{\sum_{I\in\F}z_I:
 z_I\ge0, \sum_{I\ni e}z_I\ge x_e\; \forall e\in E\right\}.
 \label{eq:gauge}
\end{equation}
Indeed, if $x\in tQ$ with $t>0$, write $x/t$ as a convex combination of
incidence vectors of members of $\F$ and set $z_I$ equal to $t$ times the
corresponding coefficients.  This gives a feasible cover of total mass
$t$.  Conversely, a cover of total mass $t>0$ gives a convex combination
$y\in Q$ with $y\ge x/t$ coordinatewise; downward closedness gives
$x/t\in Q$, and hence $x\in tQ$.  The case $x=0$ is immediate.

Dualizing the covering LP gives
\begin{equation}
 \gamma_Q(x)=\max\{q\cdot x:q\in\R_{\ge0}^E,
                         \ q(I)\le1\; \forall I\in\F\}.
 \label{eq:antiblocker}
\end{equation}
Its feasible region is the anti-blocker
$\abl(Q)=\{q\ge0:q\cdot y\le1\text{ for all }y\in Q\}$, since it suffices
to impose the inequality on the incidence vectors generating $Q$
\cite{fulkerson1971blocking}.  Thus LP
strong duality identifies \eqref{eq:gauge} with this dual maximum.  Since
$x\in\alpha Q$ exactly when $\gamma_Q(x)\le\alpha$, it remains to prove
$\gamma_Q(x)\le\alpha$ for all $x\in P$.

We will rule out a violation by taking an extreme primal--dual witness on a
counterexample with the smallest possible ground set.  Deletion and
contraction will first force every coordinate of the witness to be strict.
The tight rank constraints at the extreme point will then give a density
estimate, while complementary slackness will produce common independent
sets whose closures cover the ground set.  The density and closure estimates
together will contradict the value of the witness.

\subsection{Minimal Obstructions Have Strict Coordinates}
\label{sec:minimal}

Suppose for contradiction that $\gamma_Q(x)>\alpha$ for some $x\in P$, and choose a bad
instance with $|E|$ minimum.  An element that is a loop of one of the
matroids has coordinate zero throughout $P$ and belongs to no member of
$\F$, so it can be deleted.  We may therefore assume that every singleton
belongs to $\F$.  In particular, $\abl(Q)\subseteq[0,1]^E$ and is compact.

Choose $(x,q)\in P\times\abl(Q)$ globally maximizing $q\cdot x$.  Among
the maximizers of $q\cdot x$ in $P$, choose $x$ to be an extreme point.
Set $\tau:=q\cdot x>\alpha$.
The global choice also makes $q$ optimal in \eqref{eq:antiblocker} for
this $x$.

\begin{lemma}\label{lem:boundary}
Every coordinate of the chosen pair satisfies $q_e>0$ and $0<x_e<1$.
\end{lemma}

\begin{proof}
If $q_e=0$, restricting $x$ and $q$ to $E\setminus\{e\}$ yields feasible
vectors for the deletion of $e$ and preserves the objective $\tau$.
The same is true if $x_e=0$.  Either case contradicts the minimality of
$|E|$.  Thus $q_e>0$ and $x_e>0$ for every $e$.

It remains to exclude $x_e=1$.  The element $e$ is a nonloop in every
$M_i$ and for $x':=x|_{E\setminus\{e\}}$,
the contracted vector satisfies $x'\in\bigcap_{i=1}^kP(M_i/e)$.
Indeed, for $A\subseteq E\setminus\{e\}$,
$
 x'(A)+1=x(A\cup\{e\})\le r_i(A\cup\{e\})
          =r_{M_i/e}(A)+1
$.
Let $q':=q|_{E\setminus\{e\}}$ and define
\[
 \beta:=\max\{q'(J):J\text{ is independent in every }M_i/e\}.
\]
For every such $J$, the set $J\cup\{e\}$ lies in $\F$, so
$q_e+\beta\le1$.  The smaller ground set is nonempty, since otherwise
$\tau=q_e\le1<\alpha$.  Moreover, the displayed contracted feasibility
condition and the strict
positivity of $x'$ imply that no remaining element is a loop in any
$M_i/e$; hence its singleton is common independent.  Since $q'>0$, we
obtain $\beta>0$.

Writing $R=q'\cdot x'=\tau-q_e$ and using $\alpha>1$ gives
\[
       R>\alpha-q_e>\alpha(1-q_e)\ge\alpha\beta.
\]
Thus $q'/\beta$ is anti-blocker feasible for the contracted instance and
has value greater than $\alpha$ at $x'$, contradicting minimality.
Finally, $x_e\le1$ follows from the singleton rank inequalities, so the
lemma follows.
\end{proof}

\subsection{Tight Rank Chains and Extreme-Point Density}
\label{sec:chains}

We isolate the structural fact about one independent set polytope that drives
the density estimate.

\begin{lemma}\label{lem:chain}
Let $M$ be a matroid on $E$, let $y\in P(M)$, and suppose $y_e>0$ for all
$e\in E$.  If $\mathcal T:=\{S\subseteq E:y(S)=r_M(S)\}$,
then $\dim\operatorname{span}\{\one_S:S\in\mathcal T\}\le y(E)$.
\end{lemma}

\begin{proof}
Modularity of $y(\cdot)$, submodularity of $r_M$, and feasibility of $y$
show that $\mathcal T$ is closed under union and intersection.  Take a
maximal chain in this finite lattice,
\[
 \varnothing=S_0\subsetneq S_1\subsetneq\cdots\subsetneq S_d=U:=\bigcup_{S\in\mathcal T}S,
\]
and put $B_j:=S_j\setminus S_{j-1}$.

Every $A\in\mathcal T$ is a union of blocks $B_j$.  To see this, suppose
$A\cap B_j\ne\varnothing$.  The set
\[
      S_{j-1}\cup(A\cap S_j)=(S_{j-1}\cup A)\cap S_j
\]
belongs to $\mathcal T$ and lies strictly above $S_{j-1}$.  Since
$S_{j-1}\subsetneq S_j$ is a cover in the maximal chain, this set must be
$S_j$, and hence $B_j\subseteq A$.  It follows that the $d$ chain vectors
span all tight-set vectors.  Their nonempty successive differences make
them linearly independent, so the dimension is $d$.

Finally, $y(B_j)=r_M(S_j)-r_M(S_{j-1})$ is a positive integer: integrality follows from the rank function, and
positivity from $y_e>0$ on the nonempty block $B_j$.  Therefore
\[
        d\le\sum_{j=1}^d y(B_j)=y(U)\le y(E),
\]
as claimed.
\end{proof}

Apply \autoref{lem:chain} to $x$ in each $M_i$.  By
\autoref{lem:boundary}, neither nonnegativity nor singleton upper bounds
are active at $x$.  Since $x$ is an extreme point of $P$, the normals of
all tight rank inequalities, taken over the $k$ matroids, must span
$\R^E$.  Otherwise a sufficiently small displacement in both directions
along a nonzero vector orthogonal to all active normals would remain in
$P$, contradicting extremality.  With $X:=x(E)>0$, we obtain
\begin{equation}
       |E|\le\sum_{i=1}^k
       \dim\operatorname{span}\{\one_S:x(S)=r_i(S)\}
       \le kX.
       \label{eq:density}
\end{equation}

\subsection{Closure Aggregation}
\label{sec:closure}

Let $z$ be an optimal solution of the covering LP \eqref{eq:gauge} for
the chosen $x$.  Strong duality and the definition $\tau=q\cdot x$ give
$\sum_{I\in\F}z_I=\tau$.
Complementary slackness, together with $q_e>0$ for every element, gives
the exact marginal identities
\[
                    \sum_{I\ni e}z_I=x_e\qquad \forall e\in E,
\]
and also shows that $z_I>0$ implies $q(I)=1$.

Fix $I\in\supp(z)$.  We claim that its $k$ closures cover the ground set:
\[
                         E=\bigcup_{i=1}^k\cl_{M_i}(I).
\]
If some $e$ were outside every closure, $I\cup\{e\}$ would be independent
in every $M_i$.  But $q(I)=1$ and $q_e>0$ would give
$q(I\cup\{e\})=1+q_e>1$, violating anti-blocker feasibility.

Since $I$ is independent in $M_i$,
$x(\cl_{M_i}(I))\le r_i(\cl_{M_i}(I))=r_i(I)=|I|$.
In the sum of the displayed closure-rank inequalities over $i$, every ground element is
counted at least once by the displayed closure-cover relation, while every element of
$I$ lies in all $k$ closures.  Hence
\begin{equation}
                         X+(k-1)x(I)\le k|I|.
                         \label{eq:closure-count}
\end{equation}

We can now finish the proof of \autoref{thm:main}.  Multiply
\eqref{eq:closure-count} by $z_I$ and sum over $I\in\F$.  The exact
marginal identities displayed above imply
\begin{equation}
 \sum_Iz_I|I|=\sum_ex_e=X,
 \;
 \sum_Iz_Ix(I)=\sum_ex_e^2.
 \label{eq:two-sums}
\end{equation}
Using $\sum_{I\in\F}z_I=\tau$ and \eqref{eq:two-sums}, the aggregated closure
inequality becomes
\begin{equation}
                   \tau X+(k-1)\sum_ex_e^2\le kX.
                   \label{eq:master}
\end{equation}
At this point, Cauchy--Schwarz and \eqref{eq:density} give
\[
                  \sum_{e\in E}x_e^2\ge\frac{X^2}{|E|}\ge\frac Xk.
\]
Dividing \eqref{eq:master} by $X>0$ and substituting this estimate yields
\[
 \tau\le k-(k-1)\frac{\sum_ex_e^2}{X}
      \le k-\frac{k-1}{k}
      =k-1+\frac1k=\alpha,
\]
contrary to $\tau>\alpha$. This completes the proof of \autoref{thm:main}.

\section{A Fractional Local-Ratio Algorithm}\label{sec:algorithm}

The structural argument above also gives the algorithmic guarantee in
\autoref{cor:algorithm-intro}.  For $A\subseteq E$, let
$P_A:=\bigcap_{i=1}^kP(M_i|A)$, let
$\operatorname{LP}_A(c):=\max\{c\cdot y:y\in P_A\}$, and set
$\alpha:=k-1+1/k$.

We use fractional local ratio \cite{bar-yehuda2004local}.  Rather than round a
fixed fractional point, the algorithm repeatedly solves the LP on the current
ground set, obtains a weight layer from an optimal extreme point, subtracts
the largest multiple of that layer that preserves nonnegative residual
weights, and recurses on the positive support of the residual vector.  The
standard local-ratio induction combines the guarantees for the residual
weights and the removed layer.  Thus the problem-specific step is to find a
layer that is cheap for the fractional optimum but valuable for every maximal
common independent set.  The following lemma supplies such a layer.  Its
coefficient specializes to the one used by Anegg, Angelidakis, and Zenklusen
for hypergraph matching and $b$-matching \cite{anegg2021simpler}.

\begin{lemma}\label{lem:layer}
Let $x$ be a nonzero extreme point of $P_A$, set $X=x(A)$, and define
$d_e:=k-(k-1)x_e$.  Then $d\cdot x\le\alpha X$, and $d(I)\ge X$
for every inclusionwise maximal common independent set $I\subseteq A$.
\end{lemma}

\begin{proof}
Let $H=\{e\in A:x_e>0\}$.  The restriction $x|_H$ is extreme in $P_H$.
Moreover, every coordinate of $x|_H$ is positive, so none of the
nonnegativity constraints defining $P_H$ is active.  Every other defining
inequality is a matroid-rank inequality.
Extremality of $x|_H$ now implies that the tight rank normals
from the $k$ matroids span $\R^H$.  Applying \autoref{lem:chain} to $x|_H$
in each matroid gives $|H|\le kx(H)=kX$.

Consequently, Cauchy--Schwarz gives
$\sum_e x_e^2\ge X^2/|H|\ge X/k$, and hence
$d\cdot x=kX-(k-1)\sum_e x_e^2\le\alpha X$.  For a maximal $I$, the
closures $C_i=\cl_{M_i|A}(I)$ cover $A$.  Since
$x(C_i)\le r_i(C_i)=|I|$ and $I\subseteq C_i$ for every $i$, counting
closure multiplicities gives $X+(k-1)x(I)\le k|I|$, which is equivalent to
$d(I)\ge X$.
\end{proof}

Delete nonpositive-weight elements, which cannot improve either optimum, and
let $c\in\mathbb Q_{>0}^A$.  We implement fractional local ratio using the
layer supplied by \autoref{lem:layer}.

\begin{figure}[h]
\centering
\begin{algorithm}
\textul{\textsc{LocalRatio}$(A,c)$:}\+\\
if $A=\varnothing$, return $\varnothing$\\
find an extreme optimum $x$ of $\max\{c\cdot y:y\in P_A\}$\\
if $x(A)=0$, return $\varnothing$\\
for $e\in E$:\+ \\
$d_e=k-(k-1)x_e$ \-\\
$\lambda\gets\min_{e\in A}c_e/d_e$\\
$c'\gets c-\lambda d$\\
$B\gets \{e\in A:c'_e>0\}$\\
$S\gets \textsc{LocalRatio}(B,c'|_B)$\\
greedily extend $S$ to a maximal common independent set $I\subseteq A$\\
return $I$
\end{algorithm}
\caption{Rounding algorithm}
\label{alg:rounding}
\end{figure}

\begin{lemma}
\label{lem:local-ratio-guarantee}
For every $A\subseteq E$ and $c\in\mathbb Q_{>0}^A$, Algorithm \ref{alg:rounding} runs in polynomial time.  It returns an inclusionwise
maximal common independent set $I\subseteq A$ satisfying
$c(I)\ge\operatorname{LP}_A(c)/\alpha$.
\end{lemma}

\begin{proof}
At every nontrivial call, $0\le x_e\le1$, so $1\le d_e\le k$.
Consequently, $\lambda>0$, the residual vector $c'$ is nonnegative, and at
least one of its coordinates is zero.  Thus $B\subsetneq A$ and the recursion
is well defined.

We first prove the approximation and maximality claims by induction on $|A|$.
They are immediate for $A=\varnothing$.  Let $x$ be the extreme
optimum chosen by the algorithm and put $X:=x(A)$.  If $X=0$, positivity of
$c$ implies that no singleton is common independent; hence the empty set is
maximal and has the required value.

Suppose that the recursive output is maximal in $B$, and its greedy extension
$I$ is maximal in $A$.  Since $c'$ vanishes outside $B$ and $P_A$ is
downward closed, restriction to $B$ and zero-extension back to $A$ show that
$\operatorname{LP}_A(c')=\operatorname{LP}_B(c'|_B)$.  The induction
hypothesis gives $c'(S)\ge\operatorname{LP}_A(c')/\alpha$; because
$I\supseteq S$ and $c'\ge0$, the same bound holds for $c'(I)$.

It remains to account for the removed layer $\lambda d$.  The optimality of
$x$ for $c$, together with \autoref{lem:layer}, gives
$\operatorname{LP}_A(c)=c\cdot x\le
\operatorname{LP}_A(c')+\lambda\alpha X$.  The other half of the certificate,
$d(I)\ge X$, now yields $c(I)=c'(I)+\lambda d(I)\ge
\operatorname{LP}_A(c')/\alpha+\lambda X\ge
\operatorname{LP}_A(c)/\alpha$, completing the induction.

It remains to justify polynomial time.  Since the active set shrinks at each
call, the recursion has depth at most $|E|$.  An extreme optimum can be
obtained by lexicographically optimizing over the current optimal face.  The
polytope $P_A$ has a polynomial-time separation oracle: for each matroid,
minimizing $S\mapsto r_i(S)-y(S)$ either finds a violated rank inequality or
certifies feasibility \cite{iwata2001combinatorial}; rank queries can in turn be
implemented greedily using an independence oracle.  Exact
optimization--separation equivalence supplies the required rational optima
\cite{grotschel1981ellipsoid,schrijver2003combinatorial}, and standard
determinant bounds keep all residual weights at polynomial encoding length.
Together with the depth bound, this gives a deterministic weakly polynomial
oracle-model algorithm.
\end{proof}

For an arbitrary rational weight vector $w$, take
$A:=\{e\in E:w_e>0\}$ and $c:=w|_A$.  Downward closedness gives
$\operatorname{LP}_E(w)=\operatorname{LP}_A(c)$, so
\autoref{cor:algorithm-intro} follows from \autoref{lem:local-ratio-guarantee}.

\section{Extension to \texorpdfstring{$p$}{p}-Matchoids}\label{sec:matchoids}

We now show that the proof depends only on how many matroid constraints
contain any one element.  Let
\(
  \mathcal L=\{M_j=(E_j,\ind_j):j\in[m]\}
\)
be a collection of matroids on possibly overlapping subsets of
\(E:=\bigcup_{j=1}^mE_j\).  It defines the independence system
\begin{equation}
 \F_{\mathcal L}:=
 \{I\subseteq E:I\cap E_j\in\ind_j\text{ for every }j\in[m]\}.
 \label{eq:matchoid-family}
\end{equation}
If every element of $E$ belongs to at most $p$ of the local ground sets
$E_j$, then \eqref{eq:matchoid-family} is a \emph{$p$-matchoid}
\cite{huang2023matchoid}.  Its natural fractional and integral polytopes are
\[
 P_{\mathcal L}:=
 \{x\in[0,1]^E:x|_{E_j}\in P(M_j)\; \forall j\in[m]\},
 Q_{\mathcal L}:=\conv\{\one_I:I\in\F_{\mathcal L}\}.
\]
The intersection of $p$ matroids on $E$ is the special case $m=p$ and
$E_j=E$ for all $j$.  General rank-$p$ hypergraph $b$-matching is another
special case: the elements are the hyperedges, and every hypergraph vertex
$v$ supplies the uniform matroid of rank $b_v$ on its incident hyperedges.
Each hyperedge then occurs in at most $p$ local matroids.
Write $\alpha_p:=p-1+1/p$.

Lee, Sviridenko, and Vondr\'ak phrased their conjecture using a reduction of
a $p$-uniform matchoid to matroid $p$-parity rather than using
$P_{\mathcal L}$ directly \cite[Sections~1 and~4]{lee2013matching}.  To
describe their formulation in our notation, observe first that an element
which is a loop of a local matroid has value zero in both formulations and
may be deleted.  Next pad the representation by adding a free rank-one
matroid on $\{e\}$ for each missing occurrence of an element $e$.  Thus every
remaining element occurs in exactly $p$ local ground sets, without changing
either $\F_{\mathcal L}$ or $P_{\mathcal L}$.  Let
\[
 J(e):=\{j:e\in E_j\},\;
 V_j:=\{(j,e):e\in E_j\},\;
 V:=\mathbin{\dot\bigcup}_{j=1}^m V_j,
\]
and let $\widehat M_j$ be the copy of $M_j$ on $V_j$.  On the occurrence
ground set $V$, form the direct-sum matroid
$\widehat M:=\bigoplus_{j=1}^m\widehat M_j$.  Each original element $e$ is
represented by the $p$-set
\[
                         H_e:=\{(j,e):j\in J(e)\}.
\]
The sets $H_e$ are pairwise disjoint, and $I\subseteq E$ belongs to
$\F_{\mathcal L}$ exactly when $\bigcup_{e\in I}H_e$ is independent in
$\widehat M$.

Their LP relaxation (4.1)--(4.4) has an item variable $y_e$ for every
$H_e$ and an occurrence variable $\xi_{j,e}$ for every $(j,e)\in V$:
\begin{equation}
\label{eq:lsv-matchoid-lp}
\begin{aligned}
 \max\quad & \sum_{e\in E}w_ey_e \\
 \text{s.t.}\quad
 & \xi(S)\le r_{\widehat M}(S)
     && \forall S\subseteq V,\\
 & \xi_{j,e}=y_e
     && \forall (j,e)\in V,\\
 & \xi_{j,e}\ge0
     && \forall (j,e)\in V,\\
 & y_e\ge0
     && \forall e\in E.
\end{aligned}
\end{equation}
Their conjecture states that the integrality gap of this
LP is exactly $\alpha_p=p-1+1/p$ for the maximum-weight $p$-matchoid problem and
exactly $p-1$ for maximum-weight $p$-matroid intersection
\cite[Conjecture~1]{lee2013matching}.
Although \eqref{eq:lsv-matchoid-lp} is not written as
our intersection of local independent-set polytopes, 
the two relaxations are the same after projection.

\begin{lemma}\label{lem:matchoid-lp-equivalence}
The projection of the feasible region of \eqref{eq:lsv-matchoid-lp} onto
the $y$-coordinates is exactly $P_{\mathcal L}$.  The integral feasible
solutions correspond as well.  Consequently, for every weight vector $w$,
the fractional and integral optimum values in the two formulations agree,
and the formulations have the same integrality gap.
\end{lemma}

\begin{proof}
The rank function of the direct sum satisfies
\begin{equation}
 r_{\widehat M}(S)
 =\sum_{j=1}^m r_j\bigl(\{e\in E_j:(j,e)\in S\}\bigr)
 \qquad(S\subseteq V).
 \label{eq:direct-sum-rank}
\end{equation}
Suppose first that $(\xi,y)$ is feasible for
\eqref{eq:lsv-matchoid-lp}.  For $j\in[m]$ and $A\subseteq E_j$, apply its
rank inequality to $\widehat A_j:=\{(j,e):e\in A\}$.  The linking equalities
give
\[
 y(A)=\xi(\widehat A_j)
 \le r_{\widehat M}(\widehat A_j)=r_j(A).
\]
Hence $y|_{E_j}\in P(M_j)$ for every $j$.  The singleton rank inequalities
also give $y_e\le1$ (and give $y_e=0$ if $e$ is a loop of a local matroid),
so $y\in P_{\mathcal L}$.

Conversely, let $y\in P_{\mathcal L}$ and define
$\xi_{j,e}:=y_e$.  For any $S\subseteq V$, put
$A_j(S):=\{e\in E_j:(j,e)\in S\}$.  Feasibility in every local matroid
polytope and \eqref{eq:direct-sum-rank} give
\[
 \xi(S)=\sum_{j=1}^m y(A_j(S))
 \le\sum_{j=1}^m r_j(A_j(S))
 =r_{\widehat M}(S).
\]
Thus $(\xi,y)$ is feasible for \eqref{eq:lsv-matchoid-lp}, proving the
projection claim.  Finally, when $y$ is integral, the linking equalities
make the selected occurrence set equal to $\bigcup_{e:y_e=1}H_e$; it is
independent in $\widehat M$ exactly when
$\{e:y_e=1\}\in\F_{\mathcal L}$.  The objective is $w\cdot y$ in both
formulations, which proves the remaining assertions.
\end{proof}

By \autoref{lem:matchoid-lp-equivalence}, an upper bound for
$P_{\mathcal L}$ is an upper bound for the LP in the conjecture.  The
rank-one special case is the standard LP for $p$-uniform hypergraph matching,
whose gap is exactly $\alpha_p$ \cite{furedi1993fractional,chan2012hypergraph}.
Together with this lower bound, the next proposition resolves the
$p$-matchoid part of \cite[Conjecture~1]{lee2013matching}.  The projective-plane instances below
witness equality explicitly whenever a plane of order $p-1$ exists.

\begin{proposition}\label{prop:matchoid-containment}
Every $p$-matchoid satisfies
$P_{\mathcal L}\subseteq\alpha_pQ_{\mathcal L}$.
\end{proposition}

\begin{proof}
The proof follows \autoref{sec:gap-proof}; we record only the changes caused
by the local ground sets.  First pad the representation by adding copies of
the free rank-one matroid on $\{e\}$ until every element $e$ occurs in exactly
$p$ local ground sets.  This changes neither $\F_{\mathcal L}$ nor
$P_{\mathcal L}$.

Suppose the asserted containment fails.  Choose a minimal counterexample and
an extreme primal--dual witness $(x,q)$ exactly as in
\autoref{sec:minimal}, and write $X:=x(E)$ and
$\tau:=q\cdot x>\alpha_p$.  The boundary reduction is unchanged, except that
deletion removes an element from every local ground set containing it and
contraction contracts it in every such local matroid.  It again gives
$q_e>0$ and $0<x_e<1$ for every $e$.

The density calculation in \autoref{sec:chains} now sums over the local
matroids.  Extend each tight rank normal by zero outside its local ground set.
Extremality, \autoref{lem:chain}, and the exact frequency $p$ give
\begin{equation}
 |E|\le\sum_{j=1}^m
 \dim\operatorname{span}\{\one_S:S\subseteq E_j,\ x(S)=r_j(S)\}
 \le\sum_{j=1}^m x(E_j)=pX.
 \label{eq:matchoid-density}
\end{equation}

For the closure calculation, take the matchoid analogue of the covering-LP
solution $z$ from \autoref{sec:closure}.  Complementary slackness gives the
same marginal identities.  For $I\in\supp(z)$, the local closures
$C_j:=\cl_{M_j}(I\cap E_j)$ cover $E$: otherwise an uncovered element could
be added to $I$, contradicting $q(I)=1$ and $q_e>0$.  Each element of $I$
lies in exactly $p$ of these closures, so the local rank inequalities give
\begin{equation}
 X+(p-1)x(I)\le\sum_{j=1}^m x(C_j)
 \le\sum_{j=1}^m|I\cap E_j|=p|I|.
 \label{eq:matchoid-closure-count}
\end{equation}

The aggregation in \eqref{eq:two-sums}--\eqref{eq:master} is now verbatim,
with $p$ in place of $k$: it gives
$\tau X+(p-1)\sum_e x_e^2\le pX$.  By
\eqref{eq:matchoid-density}, Cauchy--Schwarz gives
$\sum_e x_e^2\ge X/p$, and hence $\tau\le\alpha_p$, a contradiction.
\end{proof}

The proposition immediately implies \autoref{thm:matchoid}.

\begin{lemma}\label{lem:matchoid-layer}
Let $A\subseteq E$, let $x$ be a nonzero extreme point of the natural
$p$-matchoid LP restricted to $A$, and set $X:=x(A)$ and
$d_e:=p-(p-1)x_e$.  Then $d\cdot x\le\alpha_pX$, and $d(I)\ge X$ for every
inclusionwise maximal matchoid-feasible set $I\subseteq A$.
\end{lemma}

\begin{proof}
The proof is the same as that of \autoref{lem:layer}, with local matroids in
place of the $k$ global matroids.  Pad the representation as above so that
every element occurs in exactly $p$ local ground sets.  For
$H:=\supp(x)$, the same extremality argument as in \autoref{lem:layer},
followed by \autoref{lem:chain} and the frequency count, gives
$|H|\le\sum_jx(H\cap E_j)=pX$.  Thus Cauchy--Schwarz gives
$\sum_e x_e^2\ge X/p$, and the calculation from \autoref{lem:layer} yields
$d\cdot x\le\alpha_pX$.

The other change is to replace the $k$ global closures by the closures in
the restricted local matroids.  For a maximal matchoid-feasible set $I$,
these closures cover $A$, and the exact frequency count gives
$X+(p-1)x(I)\le p|I|$, which is equivalent to $d(I)\ge X$.
\end{proof}

By \autoref{lem:matchoid-layer}, the recursive algorithm and its analysis from
\autoref{sec:algorithm} apply with $p$ in place of $k$.  Separation over the
natural matchoid LP amounts to separation over each local matroid
independent-set polytope, so the same running-time argument proves
\autoref{cor:matchoid-algorithm}.

For a hyperedge $e$ of size $p$, the certificate
$d_e=p-(p-1)x_e$ is exactly the coefficient appearing in the hypergraph
matching and $b$-matching analysis of Anegg, Angelidakis, and Zenklusen
\cite{anegg2021simpler}.  The factor is attained whenever a finite
projective plane of order $p-1$ exists, in particular whenever $p-1$ is a
prime power.  The projective-plane construction below is the classical tight
example for the standard $p$-uniform hypergraph-matching LP: F\"uredi
\cite{furedi1981maximum} gives the underlying extremal result, and F\"uredi,
Kahn, and Seymour \cite{furedi1993fractional} give the weighted extension;
Chan and Lau \cite{chan2012hypergraph} explicitly record the LP example.
We include its translation into the $p$-matchoid formulation for
completeness.

\begin{proposition}\label{prop:matchoid-tight}
If a finite projective plane of order $p-1$ exists, then the natural
$p$-matchoid LP has integrality gap $p-1+1/p$.
\end{proposition}

\begin{proof}
Let the elements be the lines of a projective plane of order $p-1$.  The
plane has $p^2-p+1$ points and the same number of lines, every point lies on
$p$ lines, every line contains $p$ points, and any two distinct lines meet.
For each point $v$, introduce the rank-one uniform matroid on the set of
lines containing $v$.  Each element belongs to exactly $p$ local matroids.

A feasible set contains at most one line, since any two lines meet, so the
integral optimum under unit weights is $1$.  On the other hand, assigning
$x_e=1/p$ to every line is LP feasible: each local rank-one constraint has
left-hand side $p(1/p)=1$.  Its value is
\[
                    \frac{p^2-p+1}{p}=p-1+\frac1p.
\]
Together with \autoref{thm:matchoid}, this proves the claim.
\end{proof}

The projective-plane examples make the density bound, the closure count,
and Cauchy--Schwarz simultaneously tight.  Consequently, an improvement to
$k-1$ for intersections of $k$ global matroids must use structure beyond
the fact that each element participates in $k$ matroid constraints and therefore requires new techniques.

\bibliographystyle{plain}
\bibliography{ref}

\end{document}